\documentclass[12pt,reqno]{amsart}
\usepackage{indentfirst, amssymb, amsmath, amsthm, mathrsfs, setspace, indentfirst, enumerate,  mathrsfs, amsmath, amsthm}
\usepackage[bookmarksnumbered, colorlinks, plainpages]{hyperref}
\usepackage{mathrsfs}
\usepackage{cite}
\usepackage{xcolor}
\usepackage{colortbl}
\usepackage{array}
\usepackage{booktabs}
\definecolor{headernavy}{RGB}{50,100,200}
\definecolor{rowtint}{RGB}{237,244,249}
\definecolor{rulecolor}{RGB}{47,84,120}
\usepackage{tikz}
\usepackage{pgfplots}
\usepgfplotslibrary{fillbetween}
\pgfplotsset{compat=1.18}
\usetikzlibrary{arrows.meta}
\usetikzlibrary{backgrounds,calc,positioning}
\usepgfplotslibrary{fillbetween}
\usepackage{float}
\usepackage{tabularx}
\usepgfplotslibrary{fillbetween}
\usepackage[table]{xcolor}
\usepackage[dvipsnames]{xcolor}
\usetikzlibrary{positioning,calc,shadows.blur}
\definecolor{headerblue}{RGB}{220,230,245}
\definecolor{lightblue}{RGB}{248,250,255}

\newtheorem*{theoA}{Theorem A}
\newtheorem*{theoB}{Theorem B}
\newtheorem*{theoC}{Theorem C}
\newtheorem*{theoD}{Theorem D}
\newtheorem*{theoE}{Theorem E}
\newtheorem*{theoF}{Theorem F}
\newtheorem*{theoG}{Theorem G}

\newtheorem*{cor A}{Corollary A}
\newtheorem*{cor B}{Corollary B}

\newtheorem{theo}{Theorem}[section]
\newtheorem{lem}{Lemma}[section]

\newtheorem{ques}{Question}[section]

\newtheorem{rem}{Remark}[section]

\newcommand{\be}{\begin{equation}}
\newcommand{\ee}{\end{equation}}
\newcommand{\beas}{\begin{eqnarray*}}
\newcommand{\eeas}{\end{eqnarray*}}
\newcommand{\bea}{\begin{eqnarray}}
\newcommand{\eea}{\end{eqnarray}}

\numberwithin{equation}{section}
\begin{document}
\title[B\MakeLowercase {ohr type inequalities with one parameter....}]{\LARGE B\Large\MakeLowercase {{ohr type inequalities with one parameter for the class of self analytic maps on unit disc}}}

\date{}
\author[]{G\MakeLowercase{outam} H\MakeLowercase{aldar}$^1$, R\MakeLowercase{akesh} S\MakeLowercase{arkar}$^{2*}$, S\MakeLowercase{ujoy} M\MakeLowercase{ajumder}$^3$  \MakeLowercase{and} A\MakeLowercase{bhijit} B\MakeLowercase{anerjee}$^4$}

\address{$^1$Department of Mathematics, Ghani Khan Choudhury Institute of Engineering and Technology, Narayanpur, Malda-732141, West Bengal, India.}
\email{goutamiitm@gmail.com, goutamiit1986@gmail.com}

\address{$^2$Department of Mathematics, Gour Mahavidyalaya, Mangalbari, Malda-732142, West Bengal, India.}
\email{rakeshsarkar.malda@gmail.com}

\address{$^3$Department of Mathematics, Raiganj University, Raiganj, West Bengal-733134, India.}
\email{sm05math@gmail.com, sjm@raiganjuniversity.ac.in}

\address{$^{4}$ Department of Mathematics, University of Kalyani, West Bengal 741235, India}
\email{abanerjee\_kal@yahoo.co.in}

\renewcommand{\thefootnote}{}
\footnote{2010 \emph{Mathematics Subject Classification}: 30H05, 30A10, 30B10, 30C45.}
\footnote{\emph{Key words and phrases}: Bohr's Radius, Analytic functions}
\footnote{*\emph{Corresponding Author}: }

\renewcommand{\thefootnote}{\arabic{footnote}}
\setcounter{footnote}{0}

\begin{abstract}
We investigate Bohr-type inequalities for bounded analytic functions, particularly in the presence of a parameter or under convex combination structures. Most of our findings are sharp, and we also generalize a number of earlier results.
\end{abstract}
\thanks{Typeset by \AmS -\LaTeX}
\maketitle

\section{\bf Introduction and Preliminaries}

The Bohr phenomenon is one of the most influential discoveries in geometric function theory and the theory of analytic functions. Originating from the study of power series by Bohr in 1914, it reveals the remarkable fact that the modulus bound of an analytic function imposes a surprisingly strong restriction on the absolute sum of its Taylor coefficients inside a strictly smaller disk. During the past century, this phenomenon has evolved into an active research area with deep connections to complex analysis, operator theory, functional analysis, harmonic mappings, several complex variables, and multidimensional holomorphic function theory.

Throughout the paper, we denote by
$\mathbb D={z\in\mathbb C:|z|<1},$
and let $\mathcal{B}(\mathbb{D})=\{f:\mathbb{D}\rightarrow \mathbb{C}:\;f(z)=\sum_{k=0}^{\infty}a_kz^k,\;\;a_k\in\mathbb{C}\;\; \text{and}\;\;|f(z)|<1\}$.   We also write $|z|=r$ throughout.

Bohr's celebrated theorem \cite{Bohr-PLMS-1914} asserts that every function
$f\in\mathcal B(\mathbb D)$ satisfies
\begin{equation}\label{e1}
	\sum_{k=0}^{\infty}|a_k|r^k\le1,
	\qquad
	r\le\frac13,
\end{equation}
where the constant $1/3$ is best possible.

Bohr originally established the inequality for $r\le1/6$, and the optimal radius $1/3$ was subsequently obtained independently by M.~Riesz, I.~Schur and F.~Wiener. This optimal constant is now universally known as the \emph{Bohr radius}, while \eqref{e1} is referred to as the classical \emph{Bohr inequality}.

Since then, the Bohr phenomenon has undergone remarkable developments in several directions. Numerous variants have been established by incorporating geometric quantities, coefficient constraints, symmetry conditions, area terms, Hardy-space norms, harmonic and quasiregular mappings, operator-theoretic techniques and higher dimensional holomorphic mappings. These extensions have substantially broadened both the scope and applicability of the classical Bohr inequality while revealing new extremal phenomena.

A considerable amount of research has been devoted to the classical Bohr radius. For detail study, we refer to the reader \cite{Abu-Ali-JMMA-2011,Dineen-Timoney-StudiaMath-1989,Bhowmik-Das-JMAA-2018,Aizenberg,Aizenberg-PAMS-1999,Abu-Ali-Ng-MathNachr-2017}, and the references therein. In particular, Ali \textit{et al.} \cite{Ali-Barnard-Solynin-2017} investigated the Bohr radius in the context of even analytic functions, while Kayumov and Ponnusamy \cite{Kayumov-aaponnusamy-CMFT-2017} examined the problem for alternating series and odd analytic functions, respectively.
\begin{table}[ht]
	\centering
	\caption{Representative sharp Bohr-type inequalities for bounded analytic functions.}
	\label{tab:BohrSurvey}
	
	\renewcommand{\arraystretch}{1.4}
	
	\begin{tabular}{|
			>{\centering\arraybackslash}m{3.0cm}|
			>{\centering\arraybackslash}m{5.2cm}|
			>{\centering\arraybackslash}m{3.8cm}|
			>{\centering\arraybackslash}m{2.0cm}|}
		
		\hline
		\rowcolor{headerblue}
		\textbf{Reference} &
		\textbf{Bohr-type inequality} &
		\textbf{Feature} &
		\textbf{Sharp radius}
		\\
		\hline
		
		Bohr (1914)
		&
		$\displaystyle
		\sum_{k=0}^{\infty}|a_k|r^k\le1$
		&
		Classical Bohr inequality
		&
		$\displaystyle \frac13$
		\\
		\hline
		
		Kayumov--Ponnusamy
		&
		$\displaystyle
		\sum_{k=0}^{\infty}|a_k|r^k+
		\frac{16}{9\pi}S_r\le1$
		&
		Area functional
		&
		$\displaystyle \frac13$
		\\
		\hline
		
		Kayumov--Ponnusamy
		&
		$\displaystyle
		|a_0|^2+
		\sum_{k=1}^{\infty}|a_k|r^k+
		\frac{9}{8\pi}S_r\le1$
		&
		Area functional with
		$|a_0|^2$
		&
		$\displaystyle \frac12$
		\\
		\hline
		
		Kayumov--Ponnusamy
		&
		$\displaystyle
		|f(z)-a_0|^2+
		\sum_{k=0}^{\infty}|a_k|r^k\le1$
		&
		Function-value correction
		&
		$\displaystyle \frac13$
		\\
		\hline
		
		Kayumov--Ponnusamy
		&
		$\displaystyle
		|f(z)|^2+
		\sum_{k=1}^{\infty}|a_k|^2r^{2k}\le1$
		&
		Quadratic version
		&
		$\displaystyle
		\sqrt{\frac{11}{27}}$
		\\
		\hline
		
		Liu \textit{et al.}
		&
		$\displaystyle
		|f(z)|+
		\sum_{k=1}^{\infty}|a_{2k}|r^{2k}\le1$
		&
		Even coefficients
		&
		$\displaystyle
		\sqrt2-1$
		\\
		\hline
		
		Ali \textit{et al.}
		&
		$\displaystyle
		\sum_{k=0}^{\infty}|a_{nk}|r^{nk}\le1$
		&
		$n$-symmetric functions
		&
		$\displaystyle
		\frac1{\sqrt[n]{3}}$
		\\
		\hline
		
		Ali \textit{et al.}
		&
		$\displaystyle
		\Bigl|
		\sum_{k=0}^{\infty}
		(-1)^ka_kr^k
		\Bigr|\le1$
		&
		Alternating series
		&
		$\displaystyle
		\frac1{\sqrt3}$
		\\
		\hline
		
	\end{tabular}
	
\end{table}
\par In 2018, kayumov and Ponnusamy \cite{Kayumov-Ponnusamy-CRM-2018} obtained the following interesting result.
\begin{theoA}{\em\cite{Kayumov-Ponnusamy-CRM-2018}}
Let $f\in\mathcal{B}(\mathbb{D})$ and $S_r$ denotes the area of the Riemann surface of the function $f^{-1}$ defined on the image of the subdisk $|z|<r$ under the mapping $f$. Then \bea\label{e1.1} \sum_{k=0}^{\infty} |a_k|r^k+\frac{16}{9\pi}S_r\leq1\;\; \text{for}\;\; r\leq \frac{1}{3},\eea and $1/3$, $16/9$ are the best possible. Moreover, \bea\label{e1.2} |a_0|^2+\sum_{k=1}^{\infty} |a_k|r^k+\frac{9}{8\pi}S_r\leq1\;\; \text{for}\;\; r\leq \frac{1}{2},\eea and the numbers $1/2$, $9/8$ cannot be improved further.
\end{theoA}
\begin{theoB}{\em\cite{Kayumov-Ponnusamy-CRM-2018}}
	Suppose $f(z)\in\mathcal{B}(\mathbb{D})$. Then \beas |a_0|+\sum_{j=1}^{\infty}\left(|a_j|+\frac{1}{2}|a_j|^2\right)r^j\leq1\;\;\text{for}\;\;r\leq\frac{1}{3},\eeas and the numbers $1/2$ and $1/3$ can not be improved.
\end{theoB}

\begin{theoC}{\em\cite{Kayumov-Ponnusamy-CRM-2018}}
	For any $f(z)\in\mathcal{B}(\mathbb{D})$, we have \beas|f(z)-a_0|^2+\sum_{k=0}^{\infty}|a_k|r^k\leq1\;\;\text{for}\;\;r\leq\frac{1}{3},\eeas and the number $1/3$ can not be improved.
\end{theoC}
\begin{theoD}{\em\cite{Kayumov-Ponnusamy-CRM-2018}}
	If $f(z)\in\mathcal{B}(\mathbb{D})$, then \beas |f(z)|^2+\sum_{k=1}^{\infty}|a_k|^2r^{2k}\leq1\;\;\text{for}\;\;r\leq\sqrt{\frac{11}{27}},\eeas and this number can not be improved.
\end{theoD}
\par In 2018, Liu \textit{et al.} \cite{Liu-Shao-Xu_JIA-2018} obtained the following interesting result.
\begin{theoE}{\em\cite{Liu-Shao-Xu_JIA-2018}}
For $f\in\mathcal{B}(\mathbb{D})$, it holds that \beas |f(z)|+\sum_{k=1}^{\infty}|a_{2k}|r^{2k}\leq1,\eeas for $r\leq\sqrt{2}-1$ and the number $\sqrt{2}-1$ can not be improved further.
\end{theoE}
\par In 2017, Kayumov and Ponnusamy \cite{Kayumov-Ponnusamy-ARXIV-2017} obtained the following result.
\begin{theoF}{\em\cite{Kayumov-Ponnusamy-ARXIV-2017}}
For $f\in\mathcal{B}(\mathbb{D})$, it holds that \beas |f(z)|+\sum_{k=N}^{\infty}|a_{k}|r^{k}\leq1 \;\;\text{for}\;\;r\leq R_N,\eeas where $R_N$ is the positive root of the equation \beas 2r^N(1+r)-(1-r)^2=0,\eeas and the number $R_N$ is the best possible. Moreover, \beas |f(z)|^2+\sum_{k=N}^{\infty}|a_{k}|r^{k}\leq1 \;\;\text{for}\;\;r\leq R_N,\eeas where $R_N^{\prime}$ is the positive root of the equation \beas r^N(1+r)-(1-r)^2=0,\eeas and the number $R_N^{\prime}$ is the best possible.
\end{theoF}
\par In 2017, Ali \textit{et al.} \cite{Ali-Barnard-Solynin-2017} obtained the following interesting results regarding $n$-symmetric function.
\begin{theoG}{\em\cite{Ali-Barnard-Solynin-2017}}
If $|\sum_{k=0}^{\infty}a_{nk}z^{nk}|\leq1$ in $\mathbb{D}$, then $\sum_{k=0}^{\infty}|a_{nk}|r^{nk}\leq1$ for all $r\leq\frac{1}{\sqrt[n]{3}}$, and the number $\frac{1}{\sqrt[n]{3}}$ is sharp.
\end{theoG}
\begin{theoF}{\em\cite{Ali-Barnard-Solynin-2017}}
If $\left|\sum_{k=0}^{\infty}a_{k}z^{k}\right|\leq1$ in $\mathbb{D}$, then \beas\Big|\sum_{k=0}^{\infty}(-1)^k|a_{k}|r^{k}\Big|\leq1\;\; \text{for all}\;\; r\leq\frac{1}{\sqrt{3}},\eeas and the number $\frac{1}{\sqrt{3}}$ is sharp.
\end{theoF}
\par It is worth emphasizing that the concept of the Bohr radius was originally introduced for analytic functions mapping the unit disk $\mathbb{D}$ into itself. Subsequently, this notion has been significantly extended to more general settings. In particular, researchers have considered mappings from $\mathbb{D}$ into the punctured disk as well as into the shifted disc $\{z:|z+\frac{\gamma}{1-\gamma}|<\frac{\gamma}{1-\gamma}\}$, where $0\leq \gamma<1$, thereby broadening the scope and applicability of the theory (see, for example, \cite{Abu-Com.Var-2010,Aizenberg-StudMat-2007,Abu-Ali-JMMA-2011,Ali-Jain-Ravichandran-2019-Results Math,Bhowmik-Das-JMAA-2018,Allu-Halder-JMAA-2021,Kumar-Complex Var-2023,Ahamed-Allu-Haldar-Ann-Acad-Fenn-2022,Kayumov-Ponnusamy-JMAA-2018,Ponnusamy-Rasila-Results Mat-2021}). Such generalizations highlight the flexibility of the Bohr phenomenon and its relevance beyond its classical formulation.
\par Moreover, the Bohr phenomenon has attracted considerable attention in several other areas of analysis. It has been extensively investigated in the context of harmonic mappings, as well as within the framework of quasi-subordination and K-quasi-regular harmonic mappings (see \cite{Kayumov-Ponnusamy-Ann. Acad.-2019,Liu-Ponnusamy-2020-RACSAM,Liu-Ponnsuamy-BMMS-2019}). In addition, important contributions have been made in the study of Hardy spaces and in various problems arising in operator theory. For detail study, we refer to the readers to \cite{Khavinson-CMFT-2004,Paulsen-Singh-PLMS-2002}, and the references therein. The phenomenon has also been explored in higher-dimensional settings, including functions of several real or complex variables (i.e., see \cite{Ahamed-Majumder-Sarkar-CAOT-2026,Aizenberg-PAMS-1999,Dineen-Timoney-StudiaMath-1989,LUIS-RACSAM-2021,Popescu-TAMS-2007,Boas-Khavinson-PAMS-1997}). These developments collectively demonstrate the wide-ranging influence and ongoing relevance of the Bohr phenomenon across different branches of mathematical analysis.
\par The principal objective of the present work is to develop a unified framework for several parameter-dependent Bohr-type inequalities associated with bounded analytic functions. Our approach not only recovers a variety of previously known results as special cases, but also yields new sharp inequalities involving convex combinations and weighted area functionals. In particular, we determine optimal parameter ranges for several families of Bohr inequalities, establish the sharpness of the corresponding constants, and identify extremal functions whenever possible. These results provide a broader perspective on the classical Bohr phenomenon and contribute new insights into the interplay between coefficient estimates, geometric quantities, and extremal problems in complex analysis.

\par Regarding Theorem A, we observe the following remark.
\begin{rem}
Let us consider the function
\bea\label{e1.4a} \phi_a(z)=\frac{a-z}{1-az}=a-(1-a^2)\sum_{k=1}^{\infty}a^{k-1}z^k,\eea where $z\in\mathbb{D}$ and $a\in[0,1/3)$. Let $\lambda$ be a non negative real number.
\par For this function, a simple calculation shows that \beas \sum_{k=0}^{\infty}|a_k|r^k+\frac{\lambda}{\pi} S_r=a+\frac{(1-a^2)r}{1-ar}+\frac{\lambda r^2(1-a^2)^2}{(1-a^2r^2)^2}.\eeas 
Substituting $r=1/3$ in the above equation, it follows that \beas \sum_{k=0}^{\infty}|a_k|r^k+\frac{\lambda}{\pi} S_r=a+\frac{(1-a^2)}{3-a}+\frac{9\lambda (1-a^2)^2}{(9-a^2)^2}.\eeas  
\par Then a straightforward calculation shows that the above expression will be less than or equal to $1$ if \beas \lambda \leq \frac{2(3-a)(3+a)^2}{9(1+a)^2}.\eeas  
\par Hence, in particular, if we set $a=1/4$ and $\lambda\in\left(\frac{16}{9},\frac{1859}{450}\right)$, then we see that \beas \sum_{k=0}^{\infty}a_kr^k+\frac{\lambda}{\pi} S_r\leq1 \;\;\text{for}\;\;r\leq\frac{1}{3}.\eeas
\end{rem}
\par Thus, from the above remark, it follows that \eqref{e1.1} still holds even if we choose $\lambda>16/9$. Naturally, this observation leads us to the following question.
\begin{ques}
What could be said about the maximum possible values of $\lambda\geq0$ so that \bea\label{e1.3} \sum_{k=0}^{\infty} |a_k|r^k+\lambda\left(\frac{S_r}{\pi}\right) \leq1\;\; \text{for}\;\; r\leq \frac{1}{3}\eea still holds?
\end{ques}
To answer this question, we prove the following theorem.
\begin{theo}\label{t1.1}
Suppose $f(z)=\sum_{k=0}^{\infty}a_kz^k$ be analytic in $\mathbb{D}=\{z\in\mathbb{C}: |z|<1\}$ with $|f(z)|\leq1$. Then \eqref{e1.3} holds under one of the following conditions:\begin{enumerate}
\item [$(a)$] $0\leq\lambda\leq\frac{16}{9}$ and $|a_0|\geq\frac{1}{3}$,\\ \item [$(b)$] $0\leq\lambda\leq\frac{4(2\sqrt{2}-\sqrt{3})}{3\sqrt{2}}$ and $|a_0|<\frac{1}{3}$,\end{enumerate} and the radius $1/3$ can not be improved further.
\end{theo}

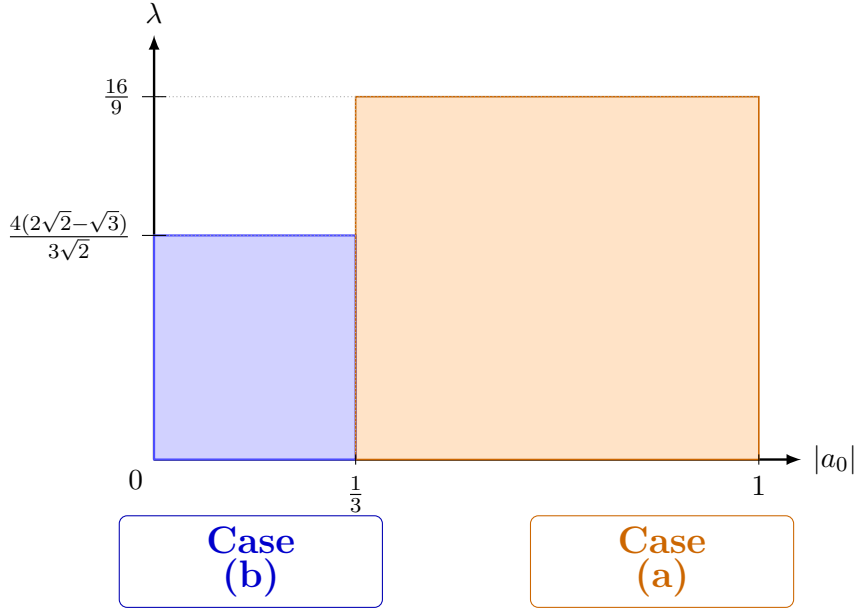
\begin{figure}[H]
	\centering
	\begin{tikzpicture}[
		x=8cm,
		y=2.7cm,
		>=latex,
		every node/.style={font=\small},
		casebox/.style={
			draw,
			rounded corners=3pt,
			blur shadow={
				shadow blur steps=4,
				shadow xshift=0.6pt,
				shadow yshift=-0.6pt,
				shadow opacity=.18},
			fill=white,
			inner sep=4pt,
			text width=3.2cm,
			align=center
		}
		]

		\def\xa{0.333333}
		\def\yb{1.098}
		\def\ya{1.777777}

		\draw[line width=0.9pt,->] (0,0)--(1.07,0)
		node[right] {$|a_0|$};
		
		\draw[line width=0.9pt,->] (0,0)--(0,2.08)
		node[above] {$\lambda$};

		\fill[blue!18]
		(0,0) rectangle (\xa,\yb);
		
		\fill[orange!22]
		(\xa,0) rectangle (1,\ya);

		\draw[blue!70,line width=.8pt]
		(0,0) rectangle (\xa,\yb);
		
		\draw[orange!80!black,line width=.6pt]
		(\xa,0) rectangle (1,\ya);

		\draw[densely dotted,gray!55]
		(\xa,0)--(\xa,\ya);
		
		\draw[densely dotted,gray!70]
		(0,\ya)--(1,\ya);
		
		\draw[densely dotted,gray!70]
		(0,\yb)--(\xa,\yb);

		\foreach \x/\lab in
		{
			0.333333/$\frac13$,
			1/$1$
		}
		{
			\draw (\x,0.025)--(\x,-0.025);
			\node[below] at (\x,-0.03){\lab};
		}
		
		\foreach \y/\lab in
		{
			1.777777/$\frac{16}{9}$,
			1.098/$\frac{4(2\sqrt2-\sqrt3)}{3\sqrt2}$
		}
		{
			\draw(0.02,\y)--(-0.02,\y);
			\node[left] at (-0.02,\y){\lab};
		}
		
		\node[below left] at (0,0){$0$};

		\node[
		casebox,
		draw=blue!70!black,
		text=blue!80!black
		] at (0.16,-0.50)
		{
			{\large\bfseries Case\\ (b)}
			
		};

		\node[
		casebox,
		draw=orange!80!black,
		text=orange!80!black
		] at (0.84,-0.50)
		{
			{\large\bfseries Case\\ (a)}
		};
		
	\end{tikzpicture}
	\vspace{.5cc}
	\caption{
		Admissible $(|a_0|,\lambda)$-region in
		Theorem~\ref{t1.1}. The estimate \eqref{e1.3}
		is valid for all points in the shaded region at the
		sharp radius $r=\frac13$.
	}
	\label{fig:region-t1.1}
\end{figure}
\begin{rem}
When $\lambda=0$, we get the classical result of Bohr's theorem.
\end{rem}
\par Inspired by the ideas and techniques underlying Theorems E and F, we aim to extend the same a more general setting. This leads us to the following result, which not only unifies Theorems E and F but also provides a natural generalization of them in a more comprehensive framework.
\begin{theo}\label{t6}
Let $f\in\mathcal{B}(\mathbb{D})$ and $\lambda\in(0,\infty)$. Then for each $m,n, N\in\mathbb{N}$, it holds that \beas |f(z^m)|+\lambda\sum_{k=N}^{\infty}|a_{nk}|r^{nk}\leq1\;\;\text{for}\;\;r\leq R_{\lambda,m,n,N},\eeas where $R_{\lambda,m,n,N}\in(0,1)$ is the positive root of the equation \beas 2\lambda r^{nN+m}+2\lambda r^{nN}-r^{n+m}+r^m+r^n-1=0,\eeas and the number $R_{\lambda,m,n,N}$ is the best possible. Also, \beas \lim\limits_{N\rightarrow \infty}R_{\lambda,m,n,N}=1\;\;\text{and}\;\;\lim\limits_{m\rightarrow\infty}R_{\lambda,m,n,N}=\alpha_N,\eeas where $\alpha_N$ is the positive root of the equation $2\lambda r^{nN}+r^n-1=0$. Moreover, \beas |f(z^m)|^2+\lambda\sum_{k=N}^{\infty}|a_{nk}|r^{nk}\leq1\;\;\text{for}\;\;r\leq R_{\lambda,m,n,N}^{\prime},\eeas where $R_{\lambda,m,n,N}^{\prime}\in(0,1)$ is unique root of the equation \beas \lambda r^{nN}(1+r^{m})-(1-r^n)(1-r^m)=0,\eeas and the number $R_{\lambda,m,n,N}^{\prime}$ can not be improved further.
\end{theo}
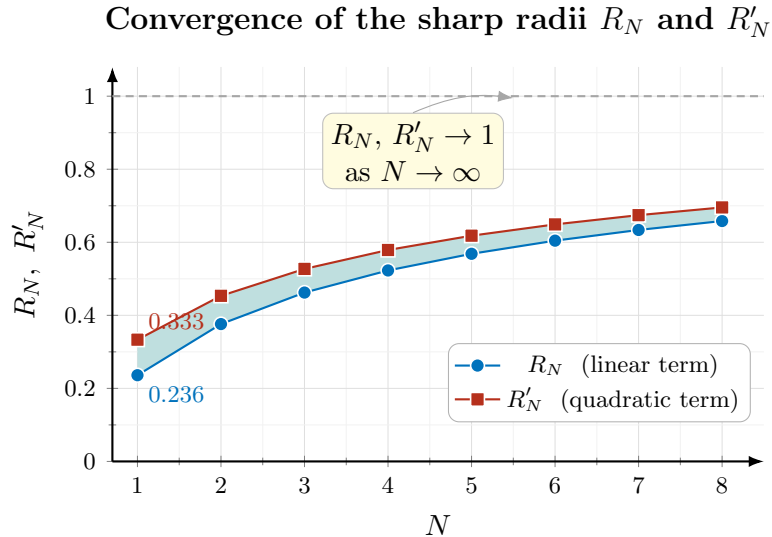
\begin{figure}[H]
	\centering
	\begin{tikzpicture}
		\begin{axis}[
			width=10.2cm,
			height=6.8cm,
			title={\bfseries Convergence of the sharp radii $R_N$ and $R_N'$},
			title style={font=\normalsize, yshift=1pt},
			xlabel={$N$},
			ylabel={$R_N,\ R_N'$},
			label style={font=\small},
			tick label style={font=\scriptsize},
			xtick={1,2,...,8},
			xmin=0.7,
			xmax=8.5,
			ymin=0,
			ymax=1.08,
			grid=both,
			major grid style={gray!25,line width=0.25pt},
			minor grid style={gray!10,line width=0.1pt},
			minor tick num=1,
			axis lines=left,
			axis line style={line width=0.8pt,-{Latex[length=2mm]}},
			every axis plot/.append style={line width=1.3pt},
			legend style={
				at={(0.98,0.30)},
				anchor=north east,
				font=\scriptsize,
				draw=gray!60,
				rounded corners,
				fill=white,
				line width=0.4pt,
				inner sep=2pt,
			},
			]
			
			\addplot[
			name path=upper, draw=none, forget plot
			] coordinates {
				(1,0.333333) (2,0.453398) (3,0.527098) (4,0.578900)
				(5,0.618034) (6,0.648989) (7,0.674277) (8,0.695438)
			};
			\addplot[
			name path=lower, draw=none, forget plot
			] coordinates {
				(1,0.236068) (2,0.376086) (3,0.462351) (4,0.522871)
				(5,0.568466) (6,0.604433) (7,0.633738) (8,0.658202)
			};
			\addplot[teal!25, forget plot] fill between[of=upper and lower];
			
				\addplot[
			thick, color=RoyalBlue, mark=*,
			mark options={fill=RoyalBlue, draw=white, line width=0.6pt},
			mark size=2.6pt,
			] coordinates {
				(1,0.236068) (2,0.376086) (3,0.462351) (4,0.522871)
				(5,0.568466) (6,0.604433) (7,0.633738) (8,0.658202)
			};
			\addlegendentry{$R_N$ \ (linear term)}
			
			\addplot[
			thick, color=BrickRed, mark=square*,
			mark options={fill=BrickRed, draw=white, line width=0.6pt},
			mark size=2.6pt,
			] coordinates {
				(1,0.333333) (2,0.453398) (3,0.527098) (4,0.578900)
				(5,0.618034) (6,0.648989) (7,0.674277) (8,0.695438)
			};
			\addlegendentry{$R_N'$ \ (quadratic term)}
			
			\draw[densely dashed, thick, gray!70] (axis cs:0.7,1) -- (axis cs:8.5,1);
			
			\node[
			font=\small, align=center, fill=yellow!15,
			rounded corners, draw=gray!50, inner sep=3pt
			] (limtxt) at (axis cs:4.3,0.85) {$R_N,\,R_N' \to 1$\\[1pt] as $N\to\infty$};
			\draw[-{Latex[length=2mm]}, gray!70]
			(limtxt.north) to[bend left=15] (axis cs:5.5,0.995);
			
			\node[font=\scriptsize, RoyalBlue, below right] at (axis cs:1,0.236068) {$0.236$};
			\node[font=\scriptsize, BrickRed, above right] at (axis cs:1,0.333333) {$0.333$};
			
		\end{axis}
	\end{tikzpicture}
	\caption{Monotone convergence of the sharp radii $R_N$ and $R_N'$ from Theorem~\ref{t6} to $1$ as $N\to\infty$. The shaded region represents the gap between the two bounds.}
	\label{fig:RN-vs-N}
\end{figure}

\par Motivated by the Theorem B, we establish the following more general result.
\begin{theo}\label{t1.3}
Let $f\in\mathcal{B}(\mathbb{D})$ and $\lambda$ be a non negative real number. Then \bea\label{e1.8} \sum_{k=0}^{\infty}|a_k|r^k+\lambda\sum_{k=0}^{\infty}|a_k|^2r^k\leq1\;\;\text{for}\;\;r\leq\frac{1}{3}\eea satisfying one of the following conditions: \begin{enumerate}
\item [{\emph{(a)}}] $0\leq\lambda\leq\frac{1}{2}\;\; \text{and}\;\;|a_0|\in[\frac{1}{3},1)$,\\
\item [{\emph{(b)}}] $0\leq\lambda\leq\frac{8-3\sqrt{2}}{4}$ and $|a_0|\in[0,\frac{1}{3})$,\end{enumerate} and the radius $1/3$ can not be improved further. 
\end{theo}
\begin{rem}
If we set $\lambda=0$, we can easily obtain the classical Theorem of Bohr Inequality. 
\end{rem}
\par In order to generalize Theorem G, we obtain the following result.
\begin{theo}\label{t1.6}
	Let $f\in\mathcal{B}(\mathbb{D})$. Then for any positive real number $\lambda$, we have \beas |a_0|+\lambda\sum_{k=1}^{\infty}|a_{nk}|r^{nk}\leq1\;\;\text{for}\;\;r\leq R_{\lambda,n},\eeas where $R_{\lambda,n}\in(0,1)$ is a positive integer of the equation \beas (2\lambda +1)r^n-1=0,\eeas and the number $R_{\lambda,n}$ can not be improved further. Moreover, \beas |a_0|^2+\lambda\sum_{k=1}^{\infty}|a_{nk}|r^{nk}\leq1\;\; \text{for}\;\;r\leq\left(\frac{1}{\lambda+1}\right)^{\frac{1}{n}},\eeas and the number $\left(\frac{1}{\lambda+1}\right)^{\frac{1}{n}}$ is the best possible.
\end{theo}
\par Inspired by Theorem F, we seek to extend it by formulating a more general version of the result as follows.
\begin{theo}\label{t1.7}
Let $f\in\mathcal{B}(\mathbb{D})$ and $\lambda\in(0,\infty)$.Then \beas\Big||a_0|+\lambda\sum_{k=1}^{\infty}(-1)^k|a_{k}|r^{k}\Big|\leq1\;\; \text{for all}\;\; r\leq\frac{1}{\sqrt{2\lambda+1}},\eeas and the number $\frac{1}{\sqrt{2\lambda+1}}$ can not be improved further.  
\end{theo}
\begin{theo}\label{t1.4}
Suppose $f\in\mathcal{B}(\mathbb{D})$ and $\lambda\in(0,\infty)$. Then \bea\label{e1.9} |f(z)-a_0|^2+|a_0|+\lambda\sum_{k=1}^{\infty}|a_k|r^k\leq1\;\;\text{for}\;\;r\leq\frac{1}{3}\eea satisfying one of the following conditions: \begin{enumerate}
\item [$(a)$] when $|a_0|\geq\frac{1}{3}$, then $\lambda\in(0,1]$,
\item [$(b)$] when $|a_0|<\frac{1}{3}$, then $\lambda\in\left(0,\frac{13\sqrt{2}}{12}\right]$.
\end{enumerate}
\end{theo}

\begin{rem}
If we put $\lambda=1$, we can easily get Theorem C.
\end{rem}
\par Motivated by \cite{Wu-Wang-Long-RACSAM-2022}, we establish the following result, which admits a convex combination interpretation and extends the earlier concept to the present framework.
\begin{theo}\label{t5}
Suppose $f\in\mathcal{B}(\mathbb{D})$. Then for $\lambda\in(0,1)$, it follows that \beas \lambda|f(z)|^2+(1-\lambda)\sum_{k=1}^{\infty}|a_k|^2r^{2k}\leq1\;\;\text{for}\;\;r\leq R_{\lambda},\eeas where \beas R_{\lambda}=\begin{cases} \frac{-2(1-\lambda)+\sqrt{4+\lambda-4\lambda^2}}{9-8\lambda},\;\;\text{when}\;\;\frac{1}{3}<\lambda<\frac{1}{2}\vspace{1.2mm}\\\text{min}\left\{\sqrt{\frac{\lambda}{1-2\lambda}},\frac{-2(1-\lambda)+\sqrt{4+\lambda-4\lambda^2}}{9-8\lambda}\right\},\;\;\text{when}\;\;0<\lambda<\frac{1}{3},\end{cases}\eeas and $R_{\lambda}\in(0,1)$ is the positive root of the equation \beas (9-8\lambda)r^2+4(1-\lambda)r-\lambda=0.\eeas \end{theo}
We conclude this section with a comparative summary of the principal results of this paper. The following table compares each theorem with existing results and lists the corresponding sharp Bohr radius together with the admissible range of $\lambda$.
	\begin{table}[H]
		\centering
		\label{tab:overview}
		
		\renewcommand{\arraystretch}{1.5}
		\setlength{\tabcolsep}{5pt}
		\arrayrulecolor{rulecolor}
		\setlength{\arrayrulewidth}{1.2pt}
		
		\rowcolors{2}{rowtint}{white}
		\begin{tabular}{|
				>{\centering\arraybackslash}m{1.65cm}|
				>{\centering\arraybackslash}m{3.75cm}|
				>{\centering\arraybackslash}m{2.3cm}|
				>{\centering\arraybackslash}m{1.5cm}|
				>{\centering\arraybackslash}m{4.35cm}|}
			\hline
			\rowcolor{headernavy}
			\textcolor{white}{\textbf{Theorem}}
			&
			\textcolor{white}{\textbf{Principal feature}}
			&
			\textcolor{white}{\textbf{Generalizes}}
			&
			\textcolor{white}{\textbf{Sharp radius}}
			&
			\textcolor{white}{\textbf{Admissible $\lambda$}}
			\\
			\hline
			
			\ref{t1.1}
			&
			Bohr inequality involving the area functional
			&
			Theorem A
			&
			$\dfrac13$
			&
			$\displaystyle
			\lambda\le
			\min\!\left\{
			\dfrac{16}{9},
			\dfrac{4(2\sqrt2-\sqrt3)}
			{3\sqrt2}
			\right\}$
			\\
			\hline
			
			\ref{t6}
			&
			Sparse-coefficient Bohr inequality for
			$|f(z^m)|$
			&
			Theorems E--F
			&
			$R_{\lambda,m,n,N}$
			&
			$\lambda>0$
			\\
			\hline
			
			\ref{t1.3}
			&
			Quadratic coefficient refinement
			&
			Theorem B
			&
			$\dfrac13$
			&
			$\displaystyle
			\lambda\le
			\min\!\left\{
			\dfrac12,
			\dfrac{8-3\sqrt2}{4}
			\right\}$
			\\
			\hline
			
			\ref{t1.6}
			&
			Bohr inequality for
			$n$-symmetric coefficients
			&
			Theorem G
			&
			$R_{\lambda,n}$
			&
			Positive root of
			$(2\lambda+1)r^n=1$
			\\
			\hline
			
			\ref{t1.7}
			&
			Alternating Bohr inequality
			&
			Classical alternating
			&
			$\dfrac1{\sqrt{2\lambda+1}}$
			&
			$\lambda>0$
			\\
			\hline
			
			\ref{t1.4}
			&
			Function-value refinement
			&
			Theorem C
			&
			$\dfrac13$
			&
			$\begin{array}{@{}c@{}}
				0<\lambda\le1,\; |a_0|\ge\dfrac13,\\[1mm]
				0<\lambda\le\dfrac{13\sqrt2}{12},\;
				|a_0|<\dfrac13
			\end{array}$
			\\
			\hline
			
			\ref{t5}
			&
			Convex-combination
			Bohr inequality
			&
			Wu--Wang--Long (2022)
			&
			$R_\lambda$
			&
			$0<\lambda<\dfrac12$
			\\
			\hline
			
		\end{tabular}\vspace{.5cc}
		\caption{Comparative summary of the principal parameter-dependent Bohr-type inequalities established in this paper.}
	\end{table}
\section{Key Lemmas}
Throughout this paper, several auxiliary results play a crucial role in establishing our main theorems. For the convenience of the reader, and to ensure clarity and continuity in the presentation, we state these lemmas below, as they will be used repeatedly in the subsequent analysis.
\begin{lem}\label{lem1.1}{\em\cite{Kayumov-aaponnusamy-CMFT-2017}}
	For any $f\in\mathcal{B}(\mathbb{D})$, we have \beas \sum_{k=1}^{\infty}|a_k|r^k\leq\begin{cases} r\frac{1-|a_0|^2}{1-r|a_0|},\;\;\text{when}\;\; |a_0|\geq r,\vspace{1.2mm}\\r\frac{\sqrt{1-|a_0|^2}}{\sqrt{1-r^2}},\;\;\text{when}\;\;|a_0|<r.\end{cases}\eeas
\end{lem}
\begin{lem}\label{lem1.2}{\em\cite{Graham-Kohr-2003}}
	For any $f\in\mathcal{B}(\mathbb{D})$, we have \beas |a_k|\leq 1-|a_0|^2\;\;\text{for all}\;\; k=1,2,3,\ldots\eeas
\end{lem}
\begin{lem}\label{lem1.3}{\em\cite{Kayumov-aaponnusamy-CMFT-2017}}
	Let $|a_0|<1$ and $0<r\leq1$. If $f\in\mathcal{B}(\mathbb{D})$, then the following sharp inequality holds: $$\sum_{j=1}^{\infty}|a_j|^2r^{pj}\leq\frac{r^p(1-|a_0|^2)^2}{1-|a_0|^2r^p}.$$
\end{lem}
\section{Proof of the Theorems}
\begin{proof}[\bf{Proof of Theorem \ref{t1.1}}]
	As left hand side of Eq. \eqref{e1.3} monotonic increasing function of $r$, it is sufficient to prove the inequality \eqref{e1.3} for $r=1/3$.
	\par First suppose that $|a_0|\geq r=1/3$. 
	\par As the left-hand side of Eq.~\eqref{e1.3} is a monotonic increasing function of $r$, in view of Lemma \ref{lem1.1} and the fact that \bea\label{e1.4} \frac{S_r}{\pi}=\sum_{k=1}^{\infty}k|a_k|^2r^{2k}\leq\frac{(1-|a_0|^2)^2r^2}{(1-r^2)^2},\eea it follows from the L.H.S. of \eqref{e1.3} that \beas \sum_{k=0}^{\infty} |a_k|r^k+\lambda\left(\frac{S_r}{\pi}\right)\leq |a_0|+\frac{1-|a_0|^2}{3-|a_0|}+\frac{9\lambda(1-|a_0|^2)^2}{64}.\eeas
	The above expression will be less than or equal to $1$, if
	\beas && 64x(3-x)+64(1-x^2)+9\lambda (3-x)(1-x^2)^2-64(3-x)\\&&=(1-x)^2[9\lambda(3-x)(1+x)^2-128]\\&&\leq 0,\eeas where $x=|a_0|\geq1/3$.
	\par Thus it is enough to prove that $\phi(x)=9\lambda(3-x)(1+x)^2-128\leq 0$, where $x\in[\frac{1}{3},1)$.
	\begin{figure}[ht]
		\centering
		\begin{tikzpicture}[scale=1]
			
			\draw[->,thick] (0,0) -- (7.2,0) node[right] {$x$};
			\draw[->,thick] (0,-2.4) -- (0,2.2) node[above] {$\phi(x)$};
		
			\draw (1,0.08)--(1,-0.08);
			\node[below] at (1,-0.08) {$\frac13$};
			
			\draw (6,0.08)--(6,-0.08);
			\node[below] at (6,-0.08) {$1$};
			
			\draw[densely dashed,gray!60] (0,0)--(7,0);
			
			\draw[
			RoyalBlue,
			very thick,
			line cap=round
			]
			(1,-1.7)
			.. controls (2.3,-1.55) and (4.2,-0.95)
			.. (6,-0.35);
			
			\fill[red!8]
			(1,-1.7)
			.. controls (2.3,-1.55) and (4.2,-0.95)
			.. (6,-0.35)
			-- (6,0)
			-- (1,0)
			-- cycle;
			
			\draw[
			RoyalBlue,
			->,
			line width=1pt
			]
			(2.2,1.25)
			to[out=-15,in=150]
			(4.6,-0.8);
			
			\node[
			draw=RoyalBlue,
			rounded corners,
			fill=blue!8,
			font=\small
			]
			at (1.85,1.45)
			{$\phi'(x)>0$};
			
			\node[
			draw=red!70,
			rounded corners,
			fill=white,
			font=\small
			]
			at (3.5,-1.05)
			{$\phi(x)\le0$};
			
			\filldraw[RoyalBlue] (6,-0.35) circle (2pt);
			
			\node[right] at (6,-0.35)
			{$\phi(1)=72\lambda-128\le0$};
			
		\end{tikzpicture}
		
		\caption{Qualitative behaviour of $\phi(x)$ on
			$\left[\frac13,1\right)$.
		}
		\label{fig:phi-proof}
	\end{figure}
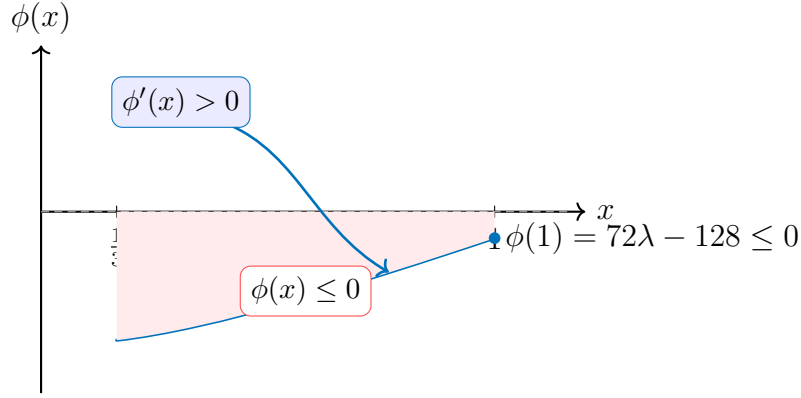

	\par Note that $\phi(x)$ is monotonic increasing in $[\frac{1}{3},1)$. Therefore, $$\phi(x)\leq\phi(1)=72\lambda-128\leq0,\;\; \text{if}\;\;\lambda\leq\frac{16}{9}.$$ 
	Next suppose that $|a_0|<r=\frac{1}{3}$. Then using Lemma \ref{lem1.1} and Eq. \eqref{e1.4} in the L.H.S. of \eqref{e1.3}, we obtain \beas \sum_{k=0}^{\infty} |a_k|r^k+\lambda\left(\frac{S_r}{\pi}\right)\leq |a_0|+\sqrt{\frac{1-|a_0|^2}{8}}+\frac{9\lambda(1-|a_0|^2)^2}{64}<\frac{1}{3}+\frac{1}{\sqrt{6}}+\frac{\lambda}{4}.\eeas
	Hence, the above expression will be less than or equal to $1$ if $\lambda\leq\frac{4(2\sqrt{2}-\sqrt{3})}{3\sqrt{2}}$, which is the condition $(b)$.
	
	\par Now to prove that the radius $1/3$ is sharp, we consider the function \bea\label{e1.6a} f_a(z)=\frac{a-z}{1-az}=a-(1-a^2)\sum_{k=1}^{\infty}a^{k-1}z^k,\eea where $a\in(0,1)$.
	\par Then we can easily deduce that \beas \sum_{k=0}^{\infty}|a_k|r^k+\lambda\left(\frac{S_r}{\pi}\right)&=&a+\frac{(1-a^2)r}{1-ar}+\frac{\lambda r^2(1-a^2)^2}{(1-a^2r^2)^2}\\&=&\frac{a(1-a^2r^2)^2+r(1-a^2)(1-ar)(1+ar)^2+\lambda r^2(1-a^2)^2}{(1-a^2r^2)^2}.\eeas 
	\par We need to prove that if $r>1/3$, then there exists an $a\in(0,1)$ such that the last expression is greater than $1$. This is equivalent to show that \beas (1-a)\left[(1-ar)(1+ar)^2(r+2ar-1)+\lambda(1-a)(1+a)^2\right]>0.\eeas  
	\par Let \beas \psi(a,\lambda,r)= (1-ar)(1+ar)^2(r+2ar-1)+\lambda(1-a)(1+a)^2.\eeas
	\par Thus, it is enough to show that if $r>1/3$, then there exists an $a\in(0,1)$ such that $\psi(a,\lambda,r)>0$.
	\par Note that $\lim_{a\rightarrow1^{-}}\psi(r)=(1-r)(1+r)^2(3r-1)>0$ for $r>1/3$. Thus $\psi(a)>0$ in a neighborhood of $1$. This completes the proof of sharpness.
\end{proof}

\begin{proof}[\bf{Proof of Theorem \ref{t6}}]
As $f(0)=a_0$, by the assumption of the theorem and using Schwarz–Pick lemma, we easily deduce that for $z=re^{i\theta}\in\mathbb{D}$, \bea\label{e3.3} |f(z)|\le\frac{r+|a_0|}{1+r|a_0|}.\eea 
Applying \eqref{e3.3} and the Lemma \ref{lem1.2} we obtain that \beas  |f(z^m)|+\lambda\sum_{k=1}^{\infty}|a_{nk}|r^{nk}&\leq&\frac{r^m+|a_0|}{1+|a_0|r^m}+\frac{\lambda r^{nN}(1-|a_0|^2)}{1-r^n}\\&=&\frac{(1-r^n)(|a_0|+r^m)+\lambda r^{nN}(1-|a_0|^2)(1+|a_0|r^m)}{(1-r^n)(1+|a_0|r^m)}.\eeas
\par We need to show that the last expression is less than or equal to $1$ for $r\leq R_{\lambda, m,n,N}$. This is equivalent to showing that \beas \phi(r)\leq0\;\;\text{for}\;\;r\leq R_{\lambda,m,n,N},\eeas where \beas \phi(r)=(1-|a_0|)\left[\lambda r^{nN}(1+|a_0|)(1+|a_0|r^m)-(1-r^n)(1-r^m)\right].\eeas Furthermore, \beas \phi(r)\leq (1-|a_0|)\left[2\lambda r^{nN}(1+r^m)-(1-r^n)(1-r^m)\right]:=(1-|a_0|)\psi(r),\eeas where \beas \psi(r)&=&2\lambda r^{nN}(1+r^m)-(1-r^n)(1-r^m)\\&=& 2\lambda r^{nN+m}+2\lambda r^{nN}-r^{n+m}+r^m+r^n-1=0.\eeas

\begin{figure}[ht]
	\centering
	
	\begin{tikzpicture}
		
		\begin{axis}[
			width=11.2cm,
			height=6.8cm,
			xmin=0,
			xmax=1.02,
			ymin=-1.2,
			ymax=9.5,
			axis lines=left,
			xlabel={$r$},
			ylabel={$\psi(r)$},
			xtick={0,0.28881,1},
			xticklabels={$0$,$R_{\lambda,m,n,N}$,$1$},
			ytick={-1,0},
			grid=major,
			major grid style={gray!18},
			tick style={black},
			axis line style={thick},
			clip=false,
			samples=250,
			domain=0:1,
			legend style={draw=none}
			]

			\addplot[
			draw=none,
			name path=A
			]
			{exp(2.4*x)-2};
			
			\addplot[
			draw=none,
			name path=B
			]
			{0};
			
			\addplot[
			red!10
			]
			fill between[
			of=A and B,
			soft clip={domain=0:0.28881}
			];

			\addplot[
			green!10
			]
			fill between[
			of=A and B,
			soft clip={domain=0.28881:1}
			];

			\addplot[
			RoyalBlue,
			very thick
			]
			{exp(2.4*x)-2};

			\addplot[
			only marks,
			RoyalBlue,
			mark=*,
			mark size=2.7pt
			]
			coordinates{
				(0.28881,0)
			};
			
			\draw[dashed,gray]
			(axis cs:0.28881,-1.2)
			--
			(axis cs:0.28881,0);

			\node[left]
			at (axis cs:0,-1)
			{$\psi(0)=-1$};
			
			\node[right]
			at (axis cs:1,9.02)
			{$\psi(1)>0$};

			\node[
			fill=white,
			rounded corners,
			draw=red!70,
			font=\small
			]
			at (axis cs:0.13,0.95)
			{$\psi(r)<0$};
			
			\node[
			fill=white,
			rounded corners,
			draw=green!60!black,
			font=\small
			]
			at (axis cs:0.70,1.10)
			{$\psi(r)>0$};

			\node[
			draw=RoyalBlue,
			fill=cyan!8,
			rounded corners,
			font=\small
			]
			(M)
			at (axis cs:0.17,6.8)
			{$\psi'(r)>0$};
			
			\draw[
			RoyalBlue,
			->,
			thick
			]
			(M.east)
			to[out=-10,in=160]
			(axis cs:0.56,4.8);

			\node[
			draw=RoyalBlue,
			fill=blue!5,
			rounded corners,
			font=\small
			]
			(C)
			at (axis cs:0.83,8.2)
			{$\psi(r)$};
			
			\draw[
			RoyalBlue,
			->,
			thick
			]
			(C.west)
			to[out=180,in=35]
			(axis cs:0.74,4.7);
			
		\end{axis}
		
	\end{tikzpicture}
	
	\caption{Schematic graph of the auxiliary function $\psi(r)$ with unique positive zero $R_{\lambda,m,n,N}$.}
	
	\label{fig:psi}
\end{figure}
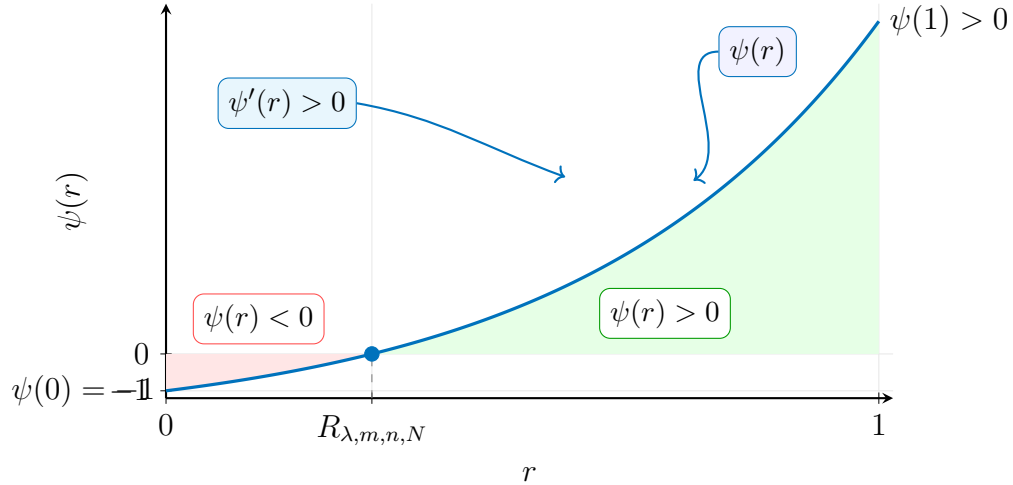
\par A direct computation shows that $\psi(r)$ in monotonically increasing in $r\in[0,1)$, $\psi(0)=-1<0$ and $\psi(1)=4\lambda>0$. Therefore, there is a unique positive root, say $R_{\lambda,m,n,N}$ of the equation \beas 2\lambda r^{nN+m}+2\lambda r^{nN}-r^{n+m}+r^m+r^n-1=0\eeas such that $\psi(r)\leq0$ for all $r\leq R_{\lambda,m,n,N}$, and hence $\phi(r)\leq0$ for all $r\leq R_{\lambda,m,n,N}$.
\par In order to prove the sharpness, we consider the same function as in \eqref{e1.4a} with $a\in[0,1)$. For this function, taking $m$ odd and $z=-r$, we easily obtain that \beas |f(-r)|+\lambda\sum_{k=N}^{\infty}|a_{nk}|r^{nk}&=&\frac{a+r^m}{1+ar^m}+\frac{\lambda a^{nN-1}r^{nN}(1-a^2)}{1-a^nr^n}\\&=& \frac{(1-a^nr^n)(a+r^m)+\lambda a^{nN-1}r^{nN}(1-a^2)(1+ar^m)}{(1+ar^m)(1-a^nr^n)}.\eeas
\par Now, It remains to prove that if $r>R_{\lambda,m,n,N}$, there exists an $a\in[0,1)$ such that the last expression is greater than $1$. This is equivalent to showing that \beas (1-a)\left[\lambda a^{nN-1}r^{nN}(1+a)(1+ar^m)-(1-a^nr^m)(1-r^m)\right]>0.\eeas
\par Thus it is enough to show that \beas P(\lambda, a, r):= \lambda a^{nN-1}r^{nN}(1+a)(1+ar^m)-(1-a^nr^m)(1-r^m)>0,\;\;\text{if}\;\;r>R_{\lambda,m,n,N}.\eeas
\par Note that \beas \lim\limits_{a\rightarrow 1^{-}}P(\lambda,a,r)=2\lambda r^{nN}(1+r^m)-(1-r^n)(1-r^m)>0,\;\;\text{when}\;\;r>R_{\lambda,m,n,N}.\eeas
\par Hence, by continuity of $P(\lambda,a,r)$, there exist $a\in(0,1)$ such that $P(\lambda,a,r)>0$ for $r>R_{\lambda,m,n,N}$. For $m$ even, the sharpness of the result can be shown similarly by taking the function $f(z)=\frac{a+z}{1+az}$, where $a\in(0,1)$ and $z\in\mathbb{D}$.
\par Next we verify the second part. A simple computation yields that \beas  |f(z^m)|^2+\lambda\sum_{k=1}^{\infty}|a_{nk}|r^{nk}&\leq&\left(\frac{r^m+|a_0|}{1+|a_0|r^m}\right)^2+\frac{\lambda r^{nN}(1-|a_0|^2)}{1-r^n}\\&=&\frac{(1-r^n)(|a_0|+r^m)^2+\lambda r^{nN}(1-|a_0|^2)(1+|a_0|r^m)^2}{(1-r^n)(1+|a_0|r^m)^2}.\eeas
\par A straightforward calculation implies that the above term is less than or equal to $1$ if \beas (1-|a_0|^2)\phi(r)\leq 0,\eeas where \beas \phi(r)=\lambda r^{nN}(1+|a_0|r^m)^2-(1-r^n)(1-r^{2m}).\eeas
Thus, it is enough to show that $\phi(r)\leq0$ for $r\leq R_{\lambda,m,n,N}^{\prime}$. Notice that \beas\phi(r)\leq(1+r^m)\left[\lambda r^{nN}(1+r^m)-(1-r^n)(1-r^m)\right]:=(1+r^m)\psi(r),\eeas where \beas\psi(r)=\lambda r^{nN}(1+r^m)-(1-r^n)(1-r^m).\eeas
\par Therefore, it is sufficient to prove that $\psi(r)\leq0$ for $r\leq R_{\lambda,m,n,N}^{\prime}$. Since $\psi(r)$ is monotonic increasing for $r\in[0,1)$, $\psi(0)=-1$ and $\psi(1)>0$. Therefore, we must have $\psi(r)\leq0$ for $r\leq R_{\lambda,m,n,N}^{\prime}$, where $R_{\lambda,m,n,N}^{\prime}$ is the positive root of the equation \beas\lambda r^{nN}(1+r^m)-(1-r^n)(1-r^m)=0.\eeas. 
\par The sharpness can be established in a manner similar to that in part one; hence, the details are omitted.
\end{proof}


\begin{proof}[\bf{Proof of Theorem \ref{t1.3}}]
By similar arguments as in Theorem \ref{t1.1}, it is enough to prove the inequality \eqref{e1.8} for $r=1/3$. First suppose that $|a_0|\geq r=1/3$.
\par Now in view of Lemmas \ref{lem1.1} and \ref{lem1.3}, it follows that \beas && \sum_{j=0}^{\infty}|a_j|r^j+\lambda\sum_{j=1}^{\infty}|a_j|^2r^j\\&&\leq|a_0|+r\frac{1-|a_0|^2}{1-r|a_0|}+\frac{\lambda  r (1-|a_0|^2)^2}{1-r|a_0|^2}\\&&= |a_0|+\frac{1-|a_0|^2}{3-|a_0|}+\frac{\lambda (1-|a_0|^2)^2}{3-|a_0|^2}\\&&=\frac{x(3-x)(3-x^2)+(1-x^2)(3-x^2)+\lambda(1-x^2)^2(3-x)}{(3-x)(3-x^2)},\eeas where $x=|a_0|\in[\frac{1}{3},1)$.
\par After some straightforward calculations, the above expression is less than or equal to $1$ if \beas\label{e1.8a} (1-x)^2\phi(x)\leq0,\eeas where \bea \phi(x)=-\lambda x^3+(\lambda+2)x^2+5\lambda x+3\lambda-6,\eea which is equivalent to showing that $\phi(x)\leq0$.
\par Now differentiating $\phi(x)$ three times, we have \beas\begin{cases}
\phi^{\prime}(x)=-3\lambda x^2+2(\lambda+2)x+5\lambda,\\\phi^{\prime\prime}(x)=-6\lambda x+2(\lambda+2),\\\phi^{\prime\prime\prime}(x)=-6\lambda\leq0.\end{cases}\eeas
\par This implies that $\phi^{\prime\prime}(x)\geq\phi^{\prime\prime}(1)=-4\lambda+4\geq0$, if $\lambda\leq1$. Therefore, for $\lambda\leq1$, we have $\phi^{\prime}(x)\geq\phi^{\prime}(\frac{1}{3})=\frac{16\lambda+4}{3}>0$. This yields that $\phi(x)\leq\phi(1)=8\lambda-4\leq0$, if $\lambda\leq\frac{1}{2}$.

\par Next suppose that $|a_0|<r=\frac{1}{3}$. Then using Lemmas \ref{lem1.1} and \ref{lem1.3}, we obtain that \beas&&\sum_{j=0}^{\infty}|a_j|r^j+\lambda\sum_{j=1}^{\infty}|a_j|^2r^j\\&&\leq|a_0|+r\frac{\sqrt{1-|a_0|^2}}{\sqrt{1-r^2}}+\frac{\lambda r(1-|a_0|^2)^2}{1-r|a_0|^2}\\&& =|a_0|+\frac{\sqrt{1-|a_0|^2}}{\sqrt{8}}+\frac{\lambda(1-|a_0|^2)^2}{3-|a_0|^2}\\&&<\frac{1}{3}+\frac{1}{\sqrt{8}}+\frac{\lambda}{3-|a_0|^2}.\eeas
\par After a straightforward calculation, it is easy to see that the last expression is less than or equal to $1$ if \beas \psi(x)= (4\sqrt{2}-3)x^2-3(4\sqrt{2}-3)+6\sqrt{2}\lambda\leq0,\eeas where $x=|a_0|$. It follows that $\psi(x)$ is monotonic decreasing in $x\in[0,\frac{1}{3})$. Therefore, we must have $\psi(x)\leq\psi(0)=9-12\sqrt{2}+6\sqrt{2}\lambda\leq0$, if $\lambda\leq\frac{4\sqrt{2-3}}{2\sqrt{2}}=\frac{8-3\sqrt{2}}{4}$. 

\par Now to prove the sharpness, we consider the same function as defined in \eqref{e1.6a}. Then after simple computations, we obtain \beas \sum_{j=0}^{\infty}|a_j|r^j+\lambda\sum_{j=1}^{\infty}|a_j|^2r^j=a+\frac{r(1-a^2)}{1-ar}+\frac{\lambda r(1-a^2)^2}{1-a^2r}.\eeas
\par It suffices to show that for $r > \frac{1}{3}$, there exists some $a \in (0,1)$ such that the above expression exceeds $1$. This is equivalent to showing that \beas \frac{(1-a)\psi(a,r,\lambda)}{(1-ar)(1-a^2r)}>0,\eeas where \beas \psi(a,r,\lambda)=\lambda r(1-a)(1+a)^2+(1+a)r-(1-ar)(1-a^2r).\eeas
\par Thus, it is enough to demonstrate that for $r>1/3$, then there exists an $a\in(0,1)$ such that $\psi(a,r,\lambda)>0$.
\par Since $r>\frac{1}{3}$, we observe that $$\lim_{a\rightarrow 1^{-}}\psi(a,r,\lambda)=-(r^2-4r+1)=-(r-2+\sqrt{3})(r-2-\sqrt{3})>0.$$ Thus, in a neighborhood of $1$, it follows that $\psi(a,r,\lambda)>0$. This completes the proof of the theorem.
\end{proof}
\begin{proof}[\bf{Proof of Theorem \ref{t1.6}}]
	In view of the Lemma \ref{lem1.2}, it follows that \beas |a_0|+\lambda\sum_{k=1}^{\infty}|a_{nk}|r^{nk}\leq\frac{|a_0|(1-r^n)+\lambda(1-|a_0|^2)r^n}{(1-r^n)}.\eeas
	The above expression is less than or equal to $1$ if $(1-|a_0|)\phi(r)\leq0$, where $\phi(r)=\lambda(1+|a_0|)r^n-(1-r^n)$. Since, $\phi(r)\leq (2\lambda+1)r^n-1:=\psi(r)$, (say); it is enough to show that $\psi(r)\leq0$ for $r\leq R_{\lambda,n}$.
	\par Note that $\psi(r)$ is monotonic increasing in $r\in(0,1)$, $\psi(0)=-1<0$ and $\psi(1)=2\lambda>0$. Therefore, there exists a positive root $R_{\lambda,n}$ in $(0,1)$ of the equation $\psi(r)=0$, and $\psi(r)\leq0$ for all $r\leq R_{\lambda,n}$.
	\par Now in order to prove the sharpness, we consider the same function as in \eqref{e1.6a}. After simple calculations, we have \beas |a_0|+\lambda\sum_{k=1}^{\infty}|a_{nk}|r^{nk}=a+\frac{\lambda(1-a^2)a^{n-1}r^n}{1-a^nr^n}.\eeas
	\par We need to show that if $r>R_{\lambda,n}$, then there exists an $a\in(0,1)$ such that the above expression is greater than $1$. This is equivalent to showing that \beas (1-a)\left[\lambda(1+a)a^{n-1}r^{n}-(1-a^nr^n)\right]>0.\eeas
	\par Therefore, our task reduces to showing that one can find an $a\in(0,1)$ for which $Q(a,\lambda,r):=\lambda(1+a)a^{n-1}r^{n}-(1-a^nr^n)>0$ whenever $r>R_{\lambda}$. Note that $Q(a,\lambda,r)$ is continuous in $a$, and \beas \lim\limits_{a\rightarrow 1^{-}}Q(a,\lambda,r)=(2\lambda+1)r^n-1>0\;\;\text{for}\;\;r>R_{\lambda,n}.\eeas
	Therefore, by continuity of $Q(a,\lambda,r)$, we have $Q(a,\lambda,r)>0$ for some $a\in(0,1)$.
\par For the second part, by applying Lemma \ref{lem1.2}, we obtain \beas |a_0|^2+\lambda\sum_{k=1}^{\infty}|a_{nk}|r^{nk}=\frac{|a_0|^2(1-r^n)+\lambda r^n(1-|a_0|^2)}{1-r^n}.\eeas
Then, by arguments similar to those used in the first part, we readily obtain the desired conclusion. Hence, we omit the details.	
\end{proof}
\begin{proof}[\bf{Proof of Theorem \ref{t1.7}}]
In view of Lemma \ref{lem1.2}, we obtain \bea\label{e3.5} M_{f}(r)&:=&|a_0|+\lambda\sum_{k=1}^{\infty}|a_{2k}|r^{2k}-\lambda\sum_{k=1}^{\infty}|a_{2k-1}|r^{2k-1}\nonumber\\&\leq& |a_0|+\lambda\sum_{k=1}^{\infty}|a_{2k}|r^{2k}\nonumber\\&\leq& |a_0|+\frac{\lambda r^2(1-|a_0|^2)}{1-r^2}.\eea	
\par The last term is less than or equal to $1$ if \beas (1-|a_0|)\phi(r)\leq 0,\eeas where $\phi(r)=(1+|a_0|)\lambda r^2-(1-r^2)$. Thus, it is enough to prove that $\phi(r)\leq0$ for $r\leq\frac{1}{\sqrt{2\lambda+1}}$. Hence \beas \phi(r)\leq(2\lambda+1)r^2-1\leq0,\;\;\text{if}\;\;r\leq\frac{1}{\sqrt{2\lambda+1}}.\eeas
\par For deriving the lower bound of $M_f(r)$, we make use of the following chain of relations. \beas  M_f(r)&:=&|a_0|+\lambda\sum_{k=1}^{\infty}|a_{2k}|r^{2k}-\lambda\sum_{k=1}^{\infty}|a_{2k-1}|r^{2k-1}\\&\geq&-\lambda\sum_{k=1}^{\infty}|a_{2k-1}|r^{2k-1}\geq-|a_0|-\sum_{k=1}^{\infty}|a_{2k-1}|r^{2k-1}\\&\geq&-\left(|a_0|+\sum_{k=1}^{\infty}|a_{2k-1}|r^{2k}\right)\\&\geq&-\left(|a_0|+\frac{\lambda(1-|a_0|^2)r^2)}{1-r^2}\right).\eeas
\par Combining this with \eqref{e3.5}, we conclude that $M_f(r)\geq-1$ for $r\leq\frac{1}{\sqrt{2\lambda+1}}$.
\par In order to establish the sharpness, we consider the following function \beas f(z)=\frac{z^2-a}{1-az^2}=-a+(1-a^2)\sum_{k=0}^{\infty}a^{k}z^{2k+2},\eeas where $a\in(0,1)$. 
\par A straightforward calculation yields that \beas M_f(r)&:=&|a_0|+\lambda\sum_{k=1}^{\infty}|a_{2k}|r^{2k}-\lambda\sum_{k=1}^{\infty}|a_{2k-1}|r^{2k-1}\\&=&a+\lambda(1-a^2)\sum_{k=1}^{\infty}a^{k-1}r^{2k}=a+\frac{\lambda(1-a^2)r^2}{1-ar^2}.\eeas
\par Now, it is sufficient to prove that if $r>\frac{1}{\sqrt{2\lambda+1}}$, there exists $a\in(0,1)$ such that the above expression is greater than $1$, which is equivalent to showing that \beas (1-a)[\lambda r^2(1+a)-(1-ar^2)]>0.\eeas
\par As $a\in(0,1)$, it is enough to prove that $P(a,\lambda,r):=\lambda r^2(1+a)-(1-ar^2)>0$.
\par We observe that $\lim\limits_{a\rightarrow 1^{-}}P(a,\lambda,r)=(2\lambda+1)r^2-1>0$ for $r>\frac{1}{\sqrt{2\lambda+1}}$. 
\par Thus, by continuity of $P(a,\lambda,r)>0$ in a neighborhood of $1$. 

\begin{figure}[H]
	\centering
	\begin{tikzpicture}
		\begin{axis}[
			width=10cm, height=6.3cm,
			xlabel={$\lambda$}, ylabel={$R_\lambda$},
			xmin=0, xmax=4, ymin=0, ymax=1.05,
			domain=0:4, samples=200,
			axis lines=left,
			axis line style={thick, black!70},
			tick label style={font=\small},
			label style={font=\small},
			grid=major, grid style={gray!12},
			minor grid style={gray!6},
			axis background/.style={fill=blue!2},
			]
			
			\addplot[
			name path=curve,
			draw=none,
			domain=0:4,
			] {1/sqrt(2*x+1)};
			\addplot[
			name path=axisbottom,
			draw=none,
			domain=0:4,
			] {0};
			\addplot[
			top color=RoyalBlue!35,
			bottom color=RoyalBlue!2,
			]
			fill between[of=curve and axisbottom];
			
			\addplot[RoyalBlue, very thick, mark=none, domain=0:4]
			{1/sqrt(2*x+1)};
			
			\addplot[
			only marks,
			mark=*,
			mark size=2.2pt,
			mark options={fill=Maroon, draw=black, line width=0.4pt}
			] coordinates {(0,1)};
			
			\node[
			font=\small,
			fill=white,
			fill opacity=0.85,
			text opacity=1,
			inner sep=2pt,
			anchor=south west
			] at (axis cs:0.15,0.95) {$R_0=1$};
			
			\draw[gray!60, densely dotted, thick] (axis cs:0,0) -- (axis cs:4,0);
			\node[
			gray!70!black,
			font=\small,
			fill=white,
			fill opacity=0.85,
			text opacity=1,
			inner sep=2pt
			] at (axis cs:3.15,0.13) {$R_\lambda\to0$ as $\lambda\to\infty$};
			
		\end{axis}
	\end{tikzpicture}
	\caption{The sharp Bohr radius $R_\lambda=1/\sqrt{2\lambda+1}$ of Theorem \ref{t1.7} as a function of $\lambda>0$.}
	\label{fig:radius-t1.7}
\end{figure}
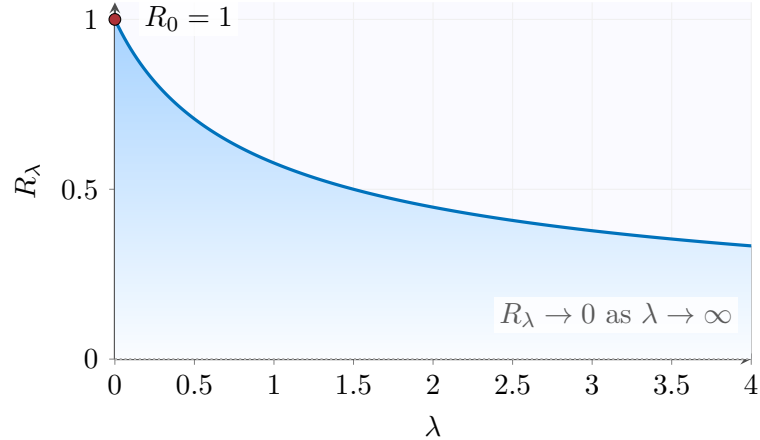
Hence, the theorem is proved.
\end{proof}

\begin{proof}[\bf{Proof of Theorem \ref{t1.4}}]
As the L.H.S. of Eq. \eqref{e1.9} is monotonic increasing in $r$, we prove that the inequality holds for $r=\frac{1}{3}$.
\par First suppose that $|a_0|\geq r=\frac{1}{3}$. Then using Lemma \ref{lem1.1}, it follows that \beas&& |f(z)-a_0|^2+|a_0|+\lambda\sum_{k=1}^{\infty}|a_k|r^k\\&&\leq|a_0|+\lambda \frac{1-|a_0|^2}{3-|a_0|}+\left(\frac{1-|a_0|^2}{3-|a_0|}\right)^2\\&&=\frac{x(3-x)^2+\lambda(3-x)(1-x^2)+(1-x^2)^2}{(3-x)^2},\eeas where $|a_0|=x\geq\frac{1}{3}$.
\par A straightforward calculation shows that the last expression is less than or equal to $1$, if \beas (1-x)\phi(x)\leq0,\eeas where $\phi(x)=-x^3-(\lambda+2)x^2+(2\lambda+7)x+3\lambda-8$. Thus it is enough to show that $\phi(x)\leq0$ for $x\in[1/3,1)$.
\par Differentiating twice, we obtain \beas \phi^{\prime}(x)=-3x^2-2(\lambda+2)x+2\lambda+7,\eeas
\beas\phi^{\prime\prime}(x)=-6x-2(\lambda+2)<0.\eeas

\par This implies that $\phi^{\prime}(x)\geq\phi^{\prime}(1)=0$. Hence $\phi^{\prime}(x)\leq\phi^{\prime}(1)=4\lambda-4\leq0$ when $\lambda\leq1$. 
\par Next suppose that $|a_0|<r=\frac{1}{3}$. Then in view of Lemma \ref{lem1.1}, we obtain \beas |f(z)-a_0|^2+|a_0|+\lambda\sum_{k=1}^{\infty}|a_k|r^k<|a_0|+\frac{\lambda\sqrt{1-|a_0|^2}}{\sqrt{8}}+\frac{1-|a_0|^2}{8}<\frac{1}{3}+\frac{\lambda}{\sqrt{8}}+\frac{1}{8}.\eeas                                                                , 
\par The above expression is less than or equal to $1$, if $\lambda\leq\frac{13\sqrt{2}}{24}$. This completes the proof.
\end{proof}

\begin{proof}[\bf{Proof of Theorem \ref{t5}}]
As $f(z)=\sum_{k=0}^{\infty}a_kz^k$ is analytic in $\mathbb{D}$ with $|f(z)|<1$ for all $z\in\mathbb{D}$, and $f(0)=a_0$, in view of the inequality \eqref{e3.3} and Lemma \ref{lem1.3}, we obtain \beas && \lambda|f(z)|^2+(1-\lambda)\sum_{k=1}^{\infty}|a_k|^2r^{2k}\\&&\leq \lambda\left(\frac{r+|a_0|}{1+r|a_0|}\right)^2+(1-\lambda)\frac{r^2(1-|a_0|^2)^2}{1-r^2|a_0|^2}\\&&=\frac{\lambda(r+|a_0|)^2(1-r|a_0|)+(1-\lambda)r^2(1-|a_0|^2)^2(1+r|a_0|)}{(1+r|a_0|)^2(1-r|a_0|)}.\eeas
\par A straightforward calculations easily shows that the above expression will be less than or equal to $1$, if $$\phi(x)\leq0,$$ where \beas \phi(x)&=&(1-\lambda)r^3x^5+(1-\lambda)r^2x^4+(2\lambda r^3-r^3-\lambda r)x^3+(2\lambda r^2-r^2+\lambda)x^2\\&+&[2\lambda r+(1-2\lambda)r^3-r]x+r^2-1,\eeas and $x=|a_0|\in[0,1)$.
\par Differentiating $\phi(x)$ four times, we have \beas \phi^{\prime}(x)&=&5(1-\lambda)r^3x^4+4(1-\lambda)r^2x^3+3(2\lambda r^3-r^3-\lambda r)x^2+2(2\lambda r^2-r^2+\lambda)x\\&+&2\lambda r+(1-2\lambda)r^3-r,\eeas
\beas \phi^{\prime\prime}(x)&=&20(1-\lambda)r^3x^3+12(1-\lambda)r^2x^2+6(2\lambda r^3-r^3-\lambda r)x+2(2\lambda r^2-r^2+\lambda),\eeas
\beas \phi^{\prime\prime\prime}(x)&=&60(1-\lambda)r^3x^2+24(1-\lambda)r^2x+6(2\lambda r^3-r^3-\lambda r),\eeas
\beas \phi^{(iv)}(x)&=&120(1-\lambda)r^3x+24(1-\lambda)r^2\geq0.\eeas
\par Therefore, we have $\phi^{\prime\prime\prime}(x)\leq\phi^{\prime\prime\prime}(1)=6r\psi(r)$, where \beas \psi(r)=(9-8\lambda)r^2+4(1-\lambda)r-\lambda.\eeas
\par We observe that $\psi(r)$ is monotonic increasing function in $[0,1)$. The zeros of $\psi(r)$ are given by \beas r_-=\frac{-2(1-\lambda)-\sqrt{4+\lambda-4\lambda^2}}{9-8\lambda}\;\;\text{and}\;\;r_+=\frac{-2(1-\lambda)+\sqrt{4+\lambda-4\lambda^2}}{9-8\lambda}.\eeas
\par Since $0<\lambda<1$, after simple computation, it follows that $r_-<0$ and $0<r_+<1$.
\par As $\psi(0)=-\lambda$, $\psi(r)$ is monotonic increasing in $[0,1)$ and $\psi(r_+)=0$, it follows that \beas \psi(r)\leq0\;\;\text{for each}\;\;r\in[0,r_+],\eeas and hence we must have $\phi^{\prime\prime\prime}(x)\leq0$ for each $x\in[0,1)$. Thus $\phi^{\prime\prime}(x)$ is monotonic decreasing in $[0,1)$. Therefore, \beas \phi^{\prime\prime}(x)\leq\phi^{\prime\prime}(0)=2(\lambda-r^2+2\lambda r^2)\leq0,\eeas whenever $r\leq\sqrt{\frac{\lambda}{1-2\lambda}}$, for each $0<\lambda<\frac{1}{2}$. Therefore, under this condition, we see that $\phi^{\prime}(x)$ is monotonic decreasing in $[0,1)$. This implies that $\phi^{\prime}(x)\leq\phi^{\prime}(0)=r(1-2\lambda)(r^2-1)\leq0$. This shows that $\phi(x)$ is monotonic decreasing in $[0,1)$, and hence $\phi(x)\leq\phi(0)=r^2-1\leq0$. Thus $\phi(x)\leq0$ for each $x\in[0,1)$.
\par We also note that $\sqrt{\frac{\lambda}{1-2\lambda}}\leq1$, whenever $\lambda\leq\frac{1}{3}$. Hence the Bohr radius $R_{\lambda}$ is given by \beas R_{\lambda}=\begin{cases} r_+,\;\;\text{when}\;\;\frac{1}{3}<\lambda\leq\frac{1}{2},\vspace{1.2mm}\\\text{min}\left\{\sqrt{\frac{\lambda}{1-2\lambda}},r_+\right\},\;\;\text{when}\;\;0<\lambda\leq\frac{1}{3}.\end{cases}\eeas

\begin{figure}[H]
	\centering
	\begin{tikzpicture}
		\begin{axis}[
			width=10.5cm, height=6.5cm,
			xlabel={$\lambda$}, ylabel={$R_\lambda$},
			xmin=0, xmax=0.5, ymin=0, ymax=1,
			domain=0.001:0.5, samples=200,
			axis lines=left,
			axis line style={thick, black!70},
			tick label style={font=\small},
			label style={font=\small},
			grid=major, grid style={gray!12},
			minor grid style={gray!8},
			axis background/.style={fill=blue!2},
			legend pos=outer north east,
			legend style={
				font=\small,
				draw=gray!40,
				rounded corners,
				fill=white,
				fill opacity=0.95,
				text opacity=1,
				row sep=4pt,
			},
			legend cell align={left},
			]
			
			\addplot[
			name path=curveA,
			draw=none,
			domain=0.001:0.3333,
			] {sqrt(x/(1-2*x))};
			\addplot[
			name path=curveB,
			draw=none,
			domain=0.3333:0.5,
			] {(-2*(1-x)+sqrt(4+x-4*x^2))/(9-8*x)};
			\addplot[
			name path=axisbottom,
			draw=none,
			domain=0.001:0.5,
			] {0};
			
			\addplot[teal!15] fill between[of=curveA and axisbottom, soft clip={domain=0.001:0.3333}];
			\addplot[orange!12] fill between[of=curveB and axisbottom, soft clip={domain=0.3333:0.5}];
			
			\addplot[RoyalBlue, very thick, mark=none, domain=0.001:0.5]
			{(-2*(1-x)+sqrt(4+x-4*x^2))/(9-8*x)};
			\addlegendentry{$r_+(\lambda)$}
			
			\addplot[Maroon, ultra thick, densely dashed, mark=none, domain=0.001:0.3333]
			{sqrt(x/(1-2*x))};
			\addlegendentry{$\sqrt{\lambda/(1-2\lambda)}$}
			
			\addplot[
			only marks,
			mark=*,
			mark size=2.3pt,
			mark options={fill=ForestGreen, draw=black, line width=0.4pt}
			] coordinates {(0.3333, 0.5774)};
			\addlegendentry{Junction at $\lambda=\tfrac13$}
			
			\draw[ForestGreen!70!black, densely dotted, thick] (axis cs:0.3333,0) -- (axis cs:0.3333,1);
			\node[ForestGreen!70!black, font=\small, fill=white, fill opacity=0.85, text opacity=1, inner sep=1.5pt] 
			at (axis cs:0.3333,0.92) {$\lambda=\tfrac13$};
			
		\end{axis}
	\end{tikzpicture}
	\caption{The two competing bounds in Theorem \ref{t5}. The sharp radius $R_\lambda$ is their pointwise minimum: $R_\lambda=\sqrt{\lambda/(1-2\lambda)}$ on $(0,1/3]$ and $R_\lambda=r_+(\lambda)$ on $(1/3,1/2]$.}
	\label{fig:radius-t5}
\end{figure}
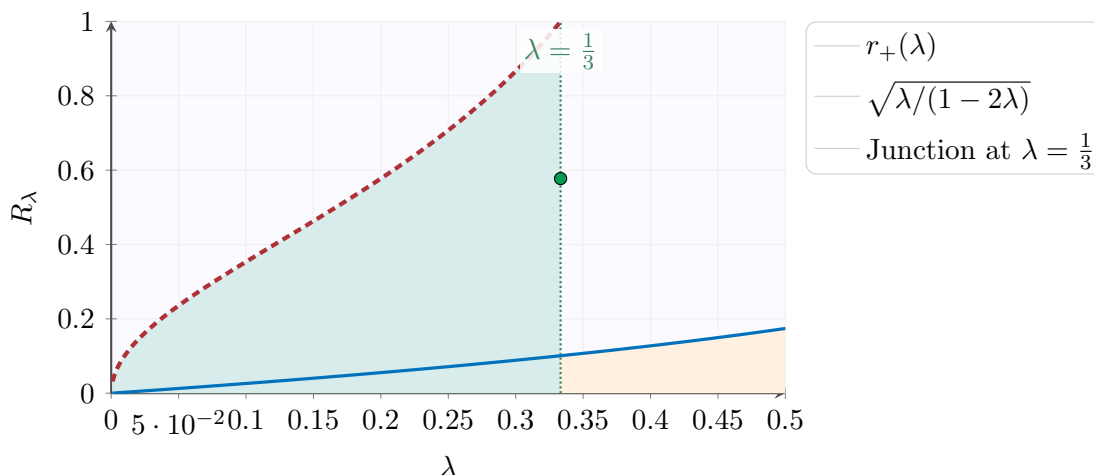

\end{proof}

\noindent\textbf{Conflict of interest:} The authors declare that there is no conflict  of interest regarding the publication of this paper.\vspace{1.2mm}

\noindent {\bf Funding:} Not Applicable.\vspace{1.2mm}

\noindent\textbf{Data availability statement:}  Data sharing not applicable to this article as no datasets were generated or analysed during the current study.\vspace{1.2mm}

\noindent {\bf Authors' contributions:} All the authors have equal contributions in preparation of the manuscript.

\end{document}